\documentclass[a4paper,reqno,11pt]{amsart}
\usepackage{amsmath,amssymb,amsbsy,amsfonts,amsthm,latexsym,
            amsopn,amstext,amsxtra,euscript,amscd,amsthm}
\newtheorem{lem}{Lemma}
\newtheorem{lemma}[lem]{Lemma}

\newtheorem{thm}{Theorem}
\newtheorem{theorem}[thm]{Theorem}

\newtheorem{cor}{Corollary}
\newtheorem{corollary}[cor]{Corollary}

\DeclareMathOperator{\mmod}{mod}

\newtheorem{defi}{Definition}
\newtheorem{definition}[defi]{Definition}

\newtheorem{rem}{Remark}
\newtheorem{remark}[rem]{Remark}

\makeatletter
\let\@@pmod\pmod
\DeclareRobustCommand{\pmod}{\@ifstar\@pmods\@@pmod}
\def\@pmods#1{\mkern4mu({\operator@font mod}\mkern 6mu#1)}
\makeatother

\def\\{\cr}
\def\({\left(}
\def\){\right)}
\def\[{\left[}
\def\]{\right]}
\def\<{\langle}
\def\>{\rangle}

\begin{document}

\title{On the discriminator of
Lucas sequences}

\author{
{\sc Bernadette Faye,}~~{\sc Florian Luca}~~ {\sc and Pieter Moree}
}

\address{{\'Ecole Doctorale de Math\'ematiques et d'Informatique}\newline
{Universit\'e Cheikh Anta Diop de Dakar}\newline
{BP 5005, Dakar Fann, S\'en\'egal}\newline
\and\newline
{AIMS-Senegal, km 2 route de Joal}\newline
{BP: 1418, Mbour, Senegal}}

\email{\tt bernadette@aims-senegal.org}

\address{{School of Mathematics, University of the Witwatersrand}\newline
{P. O. Box Wits 2050, South Africa}\newline
\and\newline
{Max Planck Institute for Mathematics}\newline
{Vivatsgasse 7, 53111 Bonn, Germany}\newline
\and\newline
{Department of Mathematics}\newline
{Faculty of Sciences}\newline
{University of Ostrava}\newline
{30 dubna 22, 701 03 Ostrava 1, Czech Republic}}
\email{\tt florian.luca@wits.ac.za}

\address{
{Max Planck Institute for Mathematics}\newline
{Vivatsgasse 7, 53111 Bonn, Germany}}
\email{\tt moree@mpim-bonn.mpg.de}

%\date{\today}

\pagenumbering{arabic}

\maketitle

\begin{abstract}
\noindent We  consider the family of Lucas sequences uniquely determined by $U_{n+2}(k)=(4k+2)U_{n+1}(k) -U_n(k),$ with initial values $U_0(k)=0$ and $U_1(k)=1$ and $k\ge 1$ an arbitrary integer. For
any integer $n\ge 1$ the discriminator function $\mathcal{D}_k(n)$ of $U_n(k)$ is defined as the smallest integer $m$ such that $U_0(k),U_1(k),\ldots,U_{n-1}(k)$ are pairwise incongruent modulo $m$.
Numerical
work of Shallit on $\mathcal{D}_k(n)$
suggests
that it has a relatively
simple characterization. In this paper
we will prove that this is indeed
the case by showing that for every $k\ge 1$ there
is a constant $n_k$ such that ${\mathcal D}_{k}(n)$
has a simple characterization for every
$n\ge n_k$. The case $k=1$ turns out
to be fundamentally different from
the case $k>1$.
\end{abstract}

\section{Introduction}
\label{intro}
The {\it discriminator} of a sequence ${\bf a}=\{a_n\}_{n\ge 1}$
of distinct integers is the sequence given by
$$
{\mathcal D}_{\bf a}(n)=\min\{m: a_0,\ldots,a_{n-1}~{\text{\rm are~pairwise~distinct~modulo}}~m\}.
$$
In other words,
${\mathcal D}_{\bf a}(n)$
is the smallest
integer $m$ that allows one to discriminate (tell
apart) the integers $a_0,\ldots, a_{n-1}$  on reducing modulo $m$.

Note that since $a_0,\ldots,a_{n-1}$ are $n$ distinct residue classes modulo ${\mathcal D}_{\bf a}(n)$ it follows that ${\mathcal D}_{\bf a}(n)\ge n$.
On the other hand obviously
$${\mathcal D}_{\bf a}(n)\le \max\{a_0,\ldots,a_{n-1}\}
-\min\{a_0,\ldots,a_{n-1}\}.$$
Put
$$
{\mathcal D}_{\bf a}=\{{\mathcal D}_{\bf a}(n): n\ge 1\}.
$$
The main problem is to give an easy description or characterization of
the discriminator (in many cases such a characterization does not seem to exist). The discriminator was named and introduced by Arnold, Benkoski and McCabe in \cite{A}.  
They considered the sequence $\bf u$ with terms $u_j = j^2$.
Meanwhile the case where $u_j=f(j)$ with $f$ a polynomial has been well-studied,
see, for example, \cite{B,M,MM,Z}. The most general result in this direction is due to Zieve \cite{Z}, who improved
on an earlier result by Moree \cite{M}.

In this paper we study the discriminator problem for
Lucas sequences (for a basic
account of Lucas sequences see, for 
example, Ribenboim
\cite[2.IV]{Ri}).
Our main results are
Theorem \ref{main} ($k=1$) and Theorem
\ref{main2} ($k>2$). Taken together with Theorem \ref{twee} ($k=2$) 
they evaluate the
discriminator for the infinite family of
second-order recurrences \eqref{basicfamily} with for each $k$ at most finitely
many not covered values.
\par All members in the family \eqref{basicfamily} have a characteristic
equation that is irreducible over
the rationals. Very recently,
Ciolan and Moree \cite{CM} determined the
discriminator for another infinite family,
this time with all members having a reducible characteristic
equation.
For every prime $q\ge 7$ they computed the discriminator of the sequence $$u_q(j)=\frac{3^j-q(-1)^{j+(q-1)/2}}{4},~j=1,2,3,\ldots$$
that was first considered in this context
by Jerzy Browkin.
The case
$q=5$ was earlier dealt with by
Moree and Zumalac\'arregui
\cite{PA}, who showed that,
for this value of $q,$ the smallest positive integer $m$
discriminating $u_q(1),\ldots,u_q(n)$ modulo $m$
equals $\min\{2^e,5^f\},$ where $e$ is the smallest integer
such that $2^e\ge n$ and $f$ is the smallest integer
such that $5^f\ge 5n/4$.
\par Despite structural similarities
between the present paper and
\cite{CM} (for example the
index of appearance $z$ in the present paper plays the
same role as the period $\rho$
in \cite{CM}), there are also many differences.
For example, Ciolan and Moree have to work
much harder to exclude small prime numbers as discriminator values.
This is related to the sequence of good
discriminator candidate
values in that case being much sparser, namely
being $O(\log x)$ for the values $\le x$,
versus $\gg \log^2 x$.
In our
case one has to work with elements and ideals in
quadratic number fields, whereas in \cite{CM}
in the proof of the main result the realm of the rationals is never left.

\par Let $k\ge 1$. For
$n\ge 0$ consider the sequence $\{U_n(k)\}_{n\ge 0}$
uniquely determined by
\begin{equation}
\label{basicfamily}
U_{n+2}(k)
=(4k+2)U_{n+1}(k)-U_n(k),~U_0(k)=0,~U_1(k)=1.
\end{equation}
For $k=1$, the sequence $\{U_n(1)\}_{n\ge 0}$ is
$$
0, 1, 6, 35, 204, 1189, 6930, 40391, 235416, 1372105, 7997214,\ldots.
$$
This is A001109 in OEIS.
On noting that
$$U_{n+2}(k)-U_{n+1}(k)=4kU_{n+1}(k)
+U_{n+1}(k)-U_n(k)\ge 1,$$
one sees that the sequence $U_n(k)$
consists of strictly increasing
non-negative numbers. Therefore we
can consider ${\mathcal D}_{U(k)}$,
which for notational convenience we
denote by ${\mathcal D}_{k}$.
\par In May 2016, Jeffrey Shallit,
who was the first to consider
${\mathcal D}_k$, wrote the third
author that numerical evidence suggests
that ${\mathcal D}_{1}(n)$
is the smallest number of the form
$250\cdot 2^i$ or $2^i$ greater than
or equal to $n$, but
that he was reluctant to conjecture such a weird thing.
More extensive numerical experiments show that
if we compute ${\mathcal D}_1(n)$ for all $n\le 2^{10}$,
then they are powers of $2$ except for $n\in [129,150]$, and other similar instances such as
$$
n\in [2^a+1,2^{a-6}\cdot 75],\text{\rm ~for~which~}{\mathcal D}_1(n)=2^{a-6}\cdot 125\text{~and~}a\in \{7,8,9\}.
$$
Thus the situation is more weird than Shallit expected
and this is confirmed by
Theorem \ref{main}.

As usual, by $\{x\}$
the
fractional part of the real number $x$ is denoted. Note that  $\{x\}=x-\lfloor x\rfloor$.
\begin{theorem}
\label{main}
Let $v_n$ be the smallest power of two
such that
$v_n\ge n$. Let $w_n$ be the smallest integer
of the form $2^a5^b$ satisfying $2^a5^b\ge 5n/3$
with $a,b\ge 1$.
Then
$$
{\mathcal D}_{1}(n)=\min\{v_n,w_n\}.
$$
Let
$$
{\mathcal M}=\left\{m\ge 1: \left\{m \frac{\log 5}{\log 2}\right\}\ge 1-\frac{\log(6/5)}{\log 2}\right\}
=\{3,6,9,12,15,18,21,\ldots\}.
$$
We have
$$\{{\mathcal D}_1(2),{\mathcal D}_1(3),{\mathcal D}_1(4),
\ldots\}=\{2^a5^b:a\ge 1,~b\in {\mathcal M}\cup \{0\}\}.$$
\end{theorem}
\noindent A straightforward application of
Weyl's criterion (cf. the proof of
\cite[Proposition 1]{PA} or \cite[Proposition 1]{CM}) gives
$$\lim_{x\rightarrow \infty}\frac{\#\{m\in {\mathcal M}:m\le x\}}{x}=\frac{\log(6/5)}{\log 2}
=0.263034\ldots.$$
\par In contrast to the case $k=1$, the
case $k=2$ turns out to be especially easy.
\begin{theorem}
\label{twee}
Let $e\ge 0$ be the smallest integer such that
$2^e\ge n$ and $f\ge 1$ the smallest  integer such that
$3\cdot 2^f\ge n$. Then
${\mathcal D}_2(n)=\min\{2^e,3\cdot 2^f\}$.
\end{theorem}
\indent Our second main result shows
that the behavior of the discriminator 
${\mathcal D}_k$ 
with $k>2$ is very different from that of
${\mathcal D}_1$.
\begin{theorem}\setcounter{thm}{1}
\label{main2}
Put
$${\mathcal A}_k=\begin{cases}
\{m~{\text{\rm odd}}:\text{\rm if~}p\mid m,~{\text{\rm then}}~p\mid k\}\text{~if~}
k\not\equiv 6\pmod*{9};\cr
\{m~{\text{\rm odd}},~9\nmid m:\text{\rm if~}p\mid m,~{\text{\rm then}}~p\mid k\}\text{~if~}
k\equiv 6\pmod*{9},
\end{cases}
$$
and $${\mathcal B}_k
=\begin{cases}
\{m~{\text{\rm even}}:\text{\rm if~}p\mid m,~{\text{\rm then}}~p\mid k(k+1)\}
\text{~if~}
k\not\equiv 2\pmod*{9};\cr
\{m~{\text{\rm even},~9\nmid m}:\text{\rm if~}p\mid m,~{\text{\rm then}}~p\mid k(k+1)\}
\text{~if~}
k\equiv 2\pmod*{9}.
\end{cases}
$$
Let $k>2$. We have
\begin{equation*}
{\mathcal D}_{k}(n)\le \min\{m\ge n:
m\in {\mathcal A}_{k}\cup {\mathcal B}_{k}\},
\end{equation*}
with equality if the interval 
$[n,3n/2)$ contains an integer 
$m\in {\mathcal A}_{k}\cup {\mathcal B}_{k}$ and with at most 
finitely many $n$ for which 
strict inequality holds.
Furthermore, we have ${\mathcal D}_{k}(n)=n$ if and only 
if $n\in {\mathcal A}_{k}\cup {\mathcal B}_{k}$.
\end{theorem}
\begin{remark}
The condition on the interval $[n,3n/2)$ is sufficient, but not always
necessary. The proof also works for $k=2$ in which
case the interval becomes $[n,5n/3)$. However,
we prefer to give a short proof from scratch of
Theorem \ref{twee} (in Section \ref{generalkintro}).
\end{remark}
Theorems \ref{twee} and \ref{main2} 
taken together have
the following corollary.
\begin{corollary}
For $k>1$ there is a finite set ${\mathcal F}_k$ such that
\begin{equation}
\label{ABF}
{\mathcal D}_{k}=
{\mathcal A}_k\cup
{\mathcal B}_k\cup {\mathcal F}_k.
\end{equation}
\end{corollary}
Note that
${\mathcal A}_1=\{1\},~
{\mathcal B}_1=\{2^e:e\ge 1\}$
and that by Theorem \ref{main} identity
\eqref{ABF} holds true with
${\mathcal F}_1=\{2^a\cdot 5^m: a\ge 1
\text{~and~} m\in {\mathcal M}\}.$
In particular, ${\mathcal F}_1$ is not
finite. In
contrast to this, Theorem \ref{twee} says that
${\mathcal F}_2$ is empty and Theorem \ref{main2} says that ${\mathcal F}_k$ is 
finite for $k>1$.
In part II \cite{CLM} the problem
of explicitly computing ${\mathcal F}_k$ is considered.
\par Despite the progress made in this paper, for most second order
recurrences (and the Fibonacci numbers belong to
this class), the discriminator remains
quite mysterious, even conjecturally. Thus in this
paper we only reveal the tip of an iceberg.
\section{Preliminaries}
\label{sec:2}

We start with some considerations about $U(k)$ valid for any $k\ge 1$. The characteristic equation of this recurrence is
$$
x^2-(4k+2)x+1=0.
$$
Its roots are $(\alpha(k),\alpha(k)^{-1})$, where
$$
\alpha(k)=2k+1+2{\sqrt{k(k+1)}}.
$$
Its discriminant is $$\Delta(k)=\Big(\alpha(k)-\frac{1}{\alpha(k)}\Big)^2=16k(k+1).$$
We have $\alpha(k)=\beta(k)^2$, where $\beta(k)={\sqrt{k+1}}+{\sqrt{k}}$. Thus,
$$
U_n(k)=\frac{\alpha(k)^{n}-\alpha(k)^{-n}}{\alpha(k)-\alpha(k)^{-1}}=\frac{\beta(k)^{2n}-\beta(k)^{-2n}}{\beta(k)^2-\beta(k)^{-2}}
$$
is both the Lucas sequence having roots $(\alpha(k),\alpha(k)^{-1})$, as well as the sequence of even indexed members of the Lehmer sequence having roots $(\beta(k),\beta(k)^{-1})$
(cf. Bilu and
Hanrot \cite {BH} or Ribenboim \cite[pp. 69-74]{Ri}). 

First we study the congruence $U_i(k)\equiv U_j(k)\pmod*{m}$ in case $m$ is an arbitrary integer. By the Chinese Remainder Theorem, it suffices to study this congruence only in the case where $m$ is a prime power.
In this section we will only deal with the easiest case
where $m$ is a power of two.
\begin{lemma}
\label{poweroftwo}
If $U_i(k)\equiv U_j(k)\pmod*{2^{a}}$, then $i\equiv j\pmod*{2^{a}}$.
\end{lemma}

\begin{proof} This is clear for $a=0$. When $a=1$, we have $U_0(k)=0,~U_1(k)=1$ and $U_{n+2}(k)\equiv -U_n(k)\pmod*{2}$. Thus, $U_{n+2}(k)\equiv U_{n}(k)\pmod*{2}$. This shows that $U_n(k)\equiv n\pmod*{2}$ for all $n\ge 0$.
Therefore $U_i(k)\equiv U_j(k)\pmod*{2}$ implies that $i\equiv j\pmod*{2}$, which is what we wanted.
We now proceed by induction on $a$. Assume
that $a>1$ and that the lemma has been proved for $a-1$. Let $i\le j$ be such that $U_i(k)\equiv U_j(k)\pmod*{2^{a}}$. In particular, $U_i(k)\equiv U_j(k)\pmod* 2$
and so $i\equiv j\pmod*{2}$.
It is easy to check that putting $V_n(k)$ for the sequence given by $V_0(k)=2,~V_1(k)=4k+2$,
we have
$$
U_j(k)-U_i(k)=U_{(j-i)/2}(k)V_{(j+i)/2}(k).
$$
The sequence $\{V_n(k)\}_{n\ge 0}$ satisfies the same recurrence  as $\{U_n(k)\}_{n\ge 0}$, namely $$V_{n+2}(k)=(4k+2)V_{n+1}(k)-V_n(k).$$
Note that $V_n(k)=\alpha(k)^n+\alpha(k)^{-n}$.
Further, by induction on $n$ using the fact that $2\| V_0(k)$ and $2\| V_1(k)$ and the recurrence for $V(k)$, we conclude that if $2\| V_n(k)$ and $2\| V_{n+1}(k)$, then
$$V_{n+2}(k)=(4k+2)V_{n+1}(k)-V_n(k)\equiv -V_n(k)\equiv 2\pmod*{4},$$
so $2\| V_{n+2}(k)$. Hence, since $2^{a}\mid U_i(k)-U_j(k)=U_{(i-j)/2}(k)V_{(i+j)/2}(k)$, and $2\| V_{(i+j)/2}(k)$, we get that $2^{a-1}\mid U_{(i-j)/2}(k)$. Thus,
$U_{(i-j)/2}(k)\equiv U_0(k)\pmod*{2^{a-1}}$ and by the  induction hypothesis we get that $(i-j)/2\equiv 0\pmod*{2^{a-1}}$. Thus, $i\equiv j\pmod*{2^{a}}$ and the induction is complete.
\end{proof}
\begin{corollary}
We have ${\mathcal D}_{k}(n)\le \min\{2^e:2^e\ge n\}$.
\end{corollary}

\section{Index of appearance}
We now need to study the congruence $U_i(k)\equiv U_j(k)\pmod*{p^{b}}$ for odd primes $p$ and integers $b\ge 1$. We start with the easy case when $j=0$. Given $m$, the smallest $n\ge 1$ such that $U_n(k)\equiv 
0\pmod*{m}$ exists, cf. \cite{BH}, and is called
the {\it index of appearance of $m$ in $U(k)$} and
is denoted by $z(m)$.
(For notational convenience we suppress
the dependence of $z(m)$ on $k$.)
The following result is well-known, cf. Bilu and Hanrot \cite{BH}. We write $\nu_p(m)$ for the exponent of the prime $p$ in the factorization of the positive integer $m$. For an odd prime $p$ we write $(\frac{\bullet}{p})$ for the Legendre symbol with respect to $p$.
\begin{lemma}
\label{lem:flauw}
The index of appearance $z$ of the
sequence $U(k)$ has the
following properties.
\begin{itemize}
\item[{\rm i)}] If $p\mid \Delta(k)$, then $z(p)=p$.
\item[{\rm ii)}] If $p\nmid \Delta(k)$, then $z(p)\mid p-e$, where $e=(\frac{\Delta(k)}{p})$.
\item[{\rm iii)}] Let $c=\nu_p(U_{z(p)}(k))$. Then $z(p^{b})=p^{\max\{b-c,0\}} z(p).$
\item[{\rm iv)}] If $p|U_m(k)$, then
$z(p)|m$.
\item[{\rm v)}] If $n=m_1\cdots m_s$ with $m_1,\ldots,m_s$ pairwise coprime, then
$$z(m_1\cdots m_s)={\text{\rm lcm}}[z(m_1),\ldots,z(m_s)].$$

\end{itemize}
\end{lemma}
Part i says that $z(p^b)=p^b$ in
case $p\mid \Delta(k)$ and
$b\ge 1$. The
next result describes what happens
for arbitrary $b$ and $p>2$.
\begin{lemma}
\label{speciaal}
Assume that $p>2$ is such that $p\mid \Delta(k)$. Let $z(p^b)$ be the index of appearance
of $p^b$ in the sequence
$U(k)$.
\begin{itemize}
\item[i)] If $p>3$, then $\nu_p(U_p)=1$. In particular, $z(p^b)=p^b$ holds for all $b\ge 1$.
\item[ii)] If $p=3$, then
$$
U_3=16k(k+1)+3.
$$
In particular, $\nu_3(U_3)=c>1$ exactly when $k\equiv 2,6\pmod* 9$. In these cases, $z(p^b)=p^{\max\{b-c,0\}}$.
Hence, $z(p^b)\mid p^{b-1}$ for all $b\ge 2$.
\end{itemize}
\end{lemma}
\begin{proof}
Recall that $\Delta(k)=16k(k+1)$.
Part i is known. As
for ii, we compute
$$
U_3=\frac{\alpha^3-\alpha^{-3}}{\alpha-\alpha^{-1}}=\alpha^2+1+\alpha^{-2}=16k(k+1)+3.
$$
Since
by assumption $3\mid 16k(k+1)$, it follows that either $3\mid k$ or $3\mid (k+1)$. In the first case, $k=3k_0$ and
$$
U_3=3(16k_0(3k_0+1)+1).
$$
The number in parenthesis is congruent to $16k_0+1\pmod*{3}$, which is a multiple of $3$ exactly when $k_0\equiv 2\pmod*{3}$; hence, $k\equiv 6\pmod*{9}$. In the second case, $k+1=3k_1$, so
$$
U_3=3(16k_1(3k_1-1)+1)
$$
and the number in parenthesis is congruent to $-16k_1+1\pmod*{3}$ which is a multiple of $3$ exactly when $k_1\equiv 1\pmod*3$, so 
$k\equiv 2\pmod*{9}$.
\end{proof}
\subsection{Index of appearance in case $k=1$}
For notational convenience we ignore
where appropriate the index $k=1$ in $U(k),~\alpha(k),~Ê\beta(k)$ and so we only write $U,~Ê\alpha,~\beta$.
We have $\Delta(1)=8$ and the relevant
quadratic field is ${\mathbb K}={\mathbb Q}[{\sqrt{2}}]$,
which has ${\mathbb Z}[{\sqrt{2}}]$
as its ring of integers.
If $\gamma,\delta\in \mathbb Z[\sqrt{2}]$,
then we write $\gamma \equiv \delta 
\pmod*{p}$
if and only if $(\gamma-\delta)/p\in \mathbb Z[\sqrt{2}]$. If $\rho=a+b\sqrt{2}\in {\mathbb K}$
with $a$ and $b$ rational numbers, then
the norm $N_{\mathbb K}(\rho)=\rho \cdot {\overline \rho}=a^2-2b^2,$ where ${\overline \rho}$ is the conjugate
of $\rho$ obtained by sending $\sqrt{2}$ to $-\sqrt{2}$.
\par For odd $p$, $z(p)$ is
a divisor of either $p-1$
or $p+1$ by Lemma \ref{lem:flauw} ii. The next lemma shows
that even more is true.
\begin{lemma}
\label{z(pb)}
Let $k=1$ and $p$ be an odd prime.
Then
$$z(p^b)\mid p^{b-1}\Big(p-\Big(\frac{2}{p}\Big)\Big)/2.$$
\end{lemma}
\begin{proof}\hfil\break
\noindent i) The case $e=(\frac{2}{p})=1$.
\medskip
\par \noindent Then $2^{(p-1)/2}\equiv 
1\pmod* p$. We have
$$
\beta^p=(1+{\sqrt{2}})^p\equiv 1+2^{p/2}\equiv 1+{\sqrt{2}} \cdot 2^{(p-1)/2} \equiv\beta\pmod*{p}.
$$
Here we used Euler's theorem that $2^{(p-1)/2}\equiv e\pmod* p$. Since $\beta$ is a unit, we
infer from $\beta^{p}\equiv \beta\pmod* p$ that
$\beta^{p-1}\equiv 1\pmod* p$. Thus,
$$
\alpha^{(p-1)/2}=(\beta^2)^{(p-1)/2}=\beta^{p-1}\equiv 1\pmod*{p}.
$$
The same congruence holds for $\alpha$ replaced by $\alpha^{-1}$. Hence, subtracting the
two congruences we get that $\alpha^{(p-1)/2}-\alpha^{-(p-1)/2}\equiv 
0\pmod* p$. Thus, $p$ divides the
difference
$\alpha^{(p-1)/2}-\alpha^{-(p-1)/2}$.
On noting that
$$N_{\mathbb K}(\alpha^{(p-1)/2}-\alpha^{-(p-1)/2})=32U_{(p-1)/2}^2,$$ we infer that
$p\mid U_{(p-1)/2}$.
Thus, $z(p)\mid (p-1)/2$, therefore $z(p^b)$ divides $p^{b-1}(p-1)/2$ by Lemma \ref{lem:flauw} iii.
\medskip \hfil\break
\noindent ii) The case $e=(\frac{2}{p})=-1$. \hfil\break
\medskip
\noindent Then $2^{(p-1)/2}\equiv -1\pmod* p$.
Now we have
$$
\beta^p=(1+{\sqrt{2}})^p\equiv 1+2^{p/2}\equiv 1+{\sqrt{2}} \cdot 2^{(p-1)/2} \equiv -\beta^{-1}\pmod*{p}.
$$
Thus, $\beta^{p+1}\equiv -1\pmod* p$. In particular,
\begin{equation}
\label{eq:minuseen}
\alpha^{(p+1)/2}=(\beta^2)^{(p+1)/2}=\beta^{p+1}\equiv -1\pmod*{p}.
\end{equation}
The same congruence holds for $\alpha$ replaced by $\alpha^{-1}$. Subtracting the
two congruences, we get that $\alpha^{(p+1)/2}-\alpha^{-(p+1)/2}\equiv 0\pmod* p$.
Noting that $$N_{\mathbb K}(\alpha^{(p+1)/2}-\alpha^{-(p+1)/2})=32U_{(p+1)/2}^2,$$ we obtain that
$p\mid U_{(p+1)/2}$.
We have,
in particular, $z(p)\mid (p+1)/2$ and
hence $z(p^b)\mid p^{b-1}(p+1)/2$
by Lemma \ref{lem:flauw} iii.
\end{proof}

Let us recall the following well-known result.

\begin{lemma}
\label{zpb}
Let $p$ be odd such that $e=(\frac{2}{p})=-1$ and
let $b\ge 1$ be an integer. Then $z(p^b)$ is the minimal $m\ge 1$ such that $\alpha^m\equiv \pm 1\pmod*{p^b}$.
\end{lemma}

\begin{proof} Assume that $m\ge 1$ is such that $\alpha^m\equiv \varepsilon \pmod*{p^b}$ for some $\varepsilon\in \{1,-1\}$. Then $\alpha^{-m}\equiv \varepsilon \pmod*{p^b}$. Subtracting both congruences we get that $p^b$ divides $\alpha^m-\alpha^{-m}$. Computing
norms
we see that
$p^{2b}\mid N_{\mathbb K}(\alpha^m-\alpha^{-m})$,
and so $p^{2b}\mid 32U_m^2,$
and therefore $p^b|U_m$, showing that $z(p^b)|m$.
Next assume that $p^b\mid U_m$ for
some $m\ge 1$. Then $\alpha^m\equiv \alpha^{-m}\pmod*{p^b}$, so $\alpha^{2m}\equiv 1\pmod*{p^b}$. Thus, $p^b\mid (\alpha^m-1)(\alpha^m+1)$.
The assumption on $e$ implies that $p$ is inert in ${\mathbb Z}[{\sqrt{2}}]$. Since, moreover,
$p$ cannot divide both $\alpha^m-1$ and $\alpha^m+1$,
it follows that $p^b$ must divide
either $\alpha^m-1$ or $\alpha^m+1$. 
\end{proof}

\section{Structure of the discriminator ${\mathcal D}_1$}

Now we are ready to restrict the number of
values the discriminator can assume.

\begin{lemma}
\label{restrict}
Let $m={\mathcal D}_1(n)$ for some $n>1$. Then
\begin{itemize}
\item[{\rm i)}] $m$ has at most one odd prime
divisor.
\item[{\rm ii)}] If $m$ is divisible by
exactly one odd prime $p$, then $e=(\frac{2}{p})=-1$ and $z(p)=(p+1)/2$.
\item[{\rm iii)}]
If $m$ is not a power of $2$, then $m$ can
be written as $2^ap^{b}$ with $a,b\ge 1$ and $p\equiv 5\pmod* 8$. 
\end{itemize}
\end{lemma}
\begin{proof}
Assume that ${\mathcal D}_1(n)=m$ and write it as
$$
m=2^a p_1^{b_1}\cdots p_r^{b_r},
$$
where the $p_i$ are distinct
odd primes.
Assume first that $r\ge 2$. Then $n\le z(m)$ (otherwise if $z(m)<n$, it follows that $U_{z(m)}\equiv 
U_0\pmod* m$, a contradiction). On recalling (Lemma \ref{lem:flauw} v) that
$
z(m)={\text{\rm lcm}}[z(2^a),z(p_1^{b_1}),\ldots,z(p_r^{b_r})],
$
we obtain the inequality
\begin{equation}
\label{zmm2}
z(m)\le  2^ap_1^{b_1-1}\cdots p_r^{b_k-1} \left(\frac{p_1+1}{2}\right)\cdots \left(\frac{p_r+1}{2}\right)<\frac{m}{2},
\end{equation}
where the last inequality needs proof.
Indeed, it is equivalent with the inequality
$$
\prod_{i=1}^r \left(\frac{p_i+1}{2}\right)<p_1\cdots p_r.
$$
It suffices to justify that
$$
\left(\frac{p_1+1}{2}\right)\left(\frac{p_2+1}{2}\right)<\frac{p_1p_2}{2}\quad {\text{\rm and}}\quad \frac{p_i+1}{2}<p_i\quad {\text{\rm for}}\quad  i=3,\ldots,r.
$$
The second inequality is clear. The first is equivalent to $p_1p_2>p_1+p_2+1$. Assuming $3\le p_1<p_2$, this inequality is implied by $p_2(p_1-2)>1$, which is obviously satisfied. \par
Since $z(m)<m/2$ by \eqref{zmm2}, it follows that the interval $[z(m),2z(m))$ contains a power of $2$, say $2^b< 2z(m)<m$. But then since $2^b\ge z(m)\ge n$, it follows that $U_0,\ldots,U_{n-1}$ are already distinct modulo $2^b$ and $2^b<m$, which contradicts the definition of the discriminator.
Thus, the only possibility is that $r\in \{0,1\}$. If $r=1$ and $e_1=(\frac{2}{p_1})=1$, then $$z(m)=z(2^ap_1^{b_1})\le 2^ap_1^{b_1-1}(p_1-1)/2<m/2,$$
and so the same contradiction applies. Assume now that $e_1=-1$ and that  $z(p_1)$ is a proper divisor of $(p+1)/2$. Then $$z(m)\le 2^a p_1^{b_1-1}z(p_1)\le 2^ap_1^{b_1-1} (p_1+1)/4<m/2,$$ and again the same contradiction applies.
\par It remains to prove part iii. We write $m=2^ap_1^{b_1}$. We know that $a\ge 1$ and $e=-1$. Thus, $p\equiv \pm 3\pmod* 8$. 
If $p\equiv 3\pmod* 8$, then
$$
z(m)={\text{\rm lcm}}[z(2^a), z(p^{b})]\mid 2^{a} p^{b-1}(p+1)/4.
$$
In particular, $z(m)<m/2$, and we get again a contradiction. Thus, $p\equiv 5\pmod* 8$. 
\end{proof}

\begin{lemma}
\label{iseven}
If $n>1$, then ${\mathcal D}_{1}(n)$ is even.
\end{lemma}

\begin{proof}
Assume that ${\mathcal D}_1(n)=m$ is odd. By the previous lemma, it follows that $m=p_1^{b_1}$, where $({2\over p_1})=-1$ and $z(p_1)=(p_1+1)/2$. Further, in this
situation \eqref{eq:minuseen} applies and
we have
$$
\alpha^{(p_1+1)/2}=-1+p_1\gamma
$$
for some algebraic integer $\gamma\in {\mathbb Z}[{\sqrt{2}}]$. By induction on $m\ge 0$ one establishes that
$$
\alpha^{p_1^m(p_1+1)/2}\equiv
-1\pmod*{p_1^{m+1}}.
$$
Let
$$
i=\left\lfloor \frac {p_1^{b_1-1}(p_1+1)}{4}\right\rfloor-1\quad {\text{\rm and}}\quad j=\frac{p_1^{b_1-1}(p_1+1)}{2}-\left(\left\lfloor \frac{p_1^{b_1-1}(p_1+1)}{4}\right\rfloor-1\right).
$$
Since $b_1\ge 1$ and $p_1\ge 3$, we have that $i\ge 0$. Further,
$$
j\ge \frac{p_1^{b_1-1}(p_1+1)}{2}-\frac{p_1^{b_1-1}(p_1+1)}{4}+1=\frac{p_1^{b_1-1}(p_1+1)}{4}+1\ge i+2,
$$
and
\begin{equation}
\label{afschatting}
j \le \frac{p_1^{b_1-1}(p_1+1)}{2}-\left(\frac{p_1^{b_1-1}(p_1+1)}{4}-\frac{3}{4}\right)+1= \frac{p_1^{b_1-1}(p_1+1)}{4}+\frac{7}{4}.
\end{equation}
Since $i+j=p_1^{b_1-1}(p_1+1)/2$, we have
$
\alpha^{i+j}\equiv -1\pmod*{p_1^{b_1}}.
$
Thus,
$$
\alpha^j\equiv -\alpha^{-i}\pmod*{p_1^{b_1}},
$$
and also
$$
\alpha^{-j}\equiv -\alpha^i\pmod*{p_1^{b_1}}.
$$
Taking the difference of the latter
two congruences we get that
$$
(\alpha^j-\alpha^{-j})-(\alpha^i-\alpha^{-i})\equiv 0\pmod*{p_1^{b_1}}.
$$
Thus, taking norms and using the fact that $p_1$ is inert in ${\mathbb K}$ and
so has norm $p_1^2$, we get
$
p_1^{2b_1}\mid N_{\mathbb K}((\alpha^j-\alpha^{-j})-(\alpha^i-\alpha^{-i}))$, that is
$$p_1^{2b_1}\mid32(U_{j}-U_i)^2,
$$
giving
$$
U_j\equiv U_i\pmod*{p_1^{b_1}}.
$$
Since $i<j$ and by assumption $U_0,\ldots,U_{n-1}$ are pairwise
distinct modulo $p_1^{b_1}$, it follows that
$j\ge n$ and hence, by \eqref{afschatting},
$$n\le \frac{p_1^{b_1-1}(p_1+1)}{4}+\frac{7}{4}.$$
We check when the right hand side is less than $m/2$. This gives
$$
\frac{p_1^{b_1-1}(p_1+1)}{4}+\frac{7}{4}<\frac{p_1^{b_1}}{2},
$$
or $2p_1^{b_1}>7+p_1^{b_1}+p_1^{b_1-1}$, which is equivalent to $p_1^{b_1-1}(p_1-1)>7$. This holds whenever  $p_1^{b_1}\ge 11$. Thus, only the cases $m=p_1^{b_1}\le 9$ need to be checked, so $n<9$. We check in this range
and we get no odd discriminant. Thus, indeed $n<m/2$, and by the previous argument we can now replace $m$ by a power of two in the interval $[m/2,m)$, and get a contradiction.
\end{proof}

\begin{lemma}
Assume that $m=2^ap_1^{b_1}={\mathcal D}_1(n)$
for some $n\ge 1$ and that $b_1\ge 1$. If $b_1>1$, then $p_1\| U_{z(p_1)}$.
\end{lemma}

\begin{proof}
This is trivial. Indeed, if $b_1>1$ and $p_1^2\mid U_{z(p_1)}$, then $z(p_1^{b_1})\mid p_1^{b_1-2}(p_1+1)/2$
by Lemma \ref{lem:flauw} iii. Thus, in this case
$$
z(m)\mid {\text{\rm lcm}}[2^a, p_1^{b_1-2}(p_1+1)/2]\mid 2^{a-1}p_1^{b_1-2} (p_1+1),$$
Since $2^{a-1}p_1^{b_1-2} (p_1+1)=m (p_1+1)/(2p_1^2)<m/2$, we have obtained
a contradiction.
\end{proof}

\begin{lemma}
\label{lem:last}
Assume that $m=2^ap_1^{b_1}$ is such that  $a\ge 1$, $p_1\equiv 5\pmod* 8$ and $z(p_1)=(p_1+1)/2$. Then
$
U_i\equiv U_j\pmod*{m}
$
holds if and only if
$i\equiv j\pmod*{z(m)}$.
\end{lemma}
\begin{proof}
Since $U_i\equiv U_j\pmod* {2^a}$, it follows that $i\equiv j\pmod* {2^a}$. It remains to understand what happens modulo $p_1^{b_1}$. Since $e_1=-1$, $p_1$ is a prime in ${\mathbb Z}[{\sqrt{2}}]$. Let $\lambda$ denote the common value of
$U_i$ and $U_j$ modulo $p_1^{b_1}$. Then $\alpha^i$ and $\alpha^j$ are both roots of
$$
x^2-4{\sqrt{2}} \lambda x-1=0\pmod*{p_1^{b_1}}
$$
in ${\mathbb Z}[{\sqrt{2}}]/(p_1^{b_1}{\mathbb Z}[{\sqrt{2}}])$. Taking the difference and factoring we get that
\begin{equation}
\label{eq:factors}
(\alpha^i-\alpha^j)(\alpha^i+\alpha^j-4{\sqrt{2}}\lambda)\equiv
0\pmod*{p_1^{b_1}}.
\end{equation}
Now various things can happen. Namely, $p_1^{b_1}$ can divide the first factor or the second factor of \eqref{eq:factors}. If $b_1>1$, some power of $p_1$ may divide the first factor and some power of $p_1$ can divide the second factor.
We investigate each of these options.\medskip
\hfil\break
\noindent i) $p_1^{b_1}\mid (\alpha^i-\alpha^j)$.
\medskip \hfil\break
\noindent Then $\alpha^{i-j}\equiv 
1\pmod*{p_1^{b_1}}$.
Since $i$ and $j$ are of the same parity, it follows that
$\alpha^{(i-j)/2}\equiv \pm 1\pmod*{p_1^{b_1}}$.
By Lemma \ref{zpb} we then infer that
$z(p_1^{b_1})|(i-j)/2$.
By Lemma \ref{z(pb)} we have
$z(p_1^{b_1})|p_1^{b_1-1}(p_1+1)/2$.
Since by assumption $p_1\equiv 5\pmod* 8$ it
follows that $z(p_1^{b_1})$ is odd
and so divides $i-j$.
Since $i-j$ is also divisible by $2^a=z(2^a)$, it is
divisible by ${\text{\rm lcm}}[z(2^a), z(p_1^{b_1})]=z(m)$.\medskip \hfil\break
\noindent ii) $p_1^{b_1}$ does not divide $\alpha^i-\alpha^j$.
\medskip
\hfil\break
\noindent We want to show that this case does not occur. If it does, then $p_1$ divides
\begin{equation}
\label{eq:7}
\alpha^i+\alpha^j-4{\sqrt{2}}\lambda.
\end{equation}
Assume first that $p_1\mid \lambda$. Then $p_1\mid U_i$ and $p_1\mid U_j$ so both $i$ and $j$ are divisible by  the
odd number $z(p_1)=(p_1+1)/2$. Also, $i\equiv j\pmod*{2}$. Since
$i=z(p_1)i_1$ and $j=z(p_1)j_1$, where
$i_1\equiv j_1\pmod*{2}$ and $\alpha^{z(p_1)}\equiv -1\pmod*{p_1}$, it follows that $\alpha^i$ and $\alpha^j$ are both congruent either to $1$ (if $i_1$ and $j_1$ are even) or to $-1$ (if $i_1$ and $j_1$ are odd) modulo $p_1$.
Thus, modulo $p_1$ the expression \eqref{eq:7} is in fact congruent to $\pm 2$ modulo $p_1$, which is certainly not zero. Thus, $\lambda\ne 0$. Then
\begin{equation}
\label{eq:een}
\alpha^i+\alpha^j\equiv 4{\sqrt{2}}\lambda
\pmod*{p_1}.
\end{equation}
The prime $p_1$ is inert so we can conjugate the above relation to get
\begin{equation}
\label{eq:twee}
\alpha^{-i}+\alpha^{-j}\equiv -4{\sqrt{2}}\lambda\pmod*{p_1}.
\end{equation}
Multiplying the second congruence by $\alpha^{i+j}$ and subtracting
\eqref{eq:twee} from \eqref{eq:een}, we get $4{\sqrt{2}} \lambda(\alpha^{i+j}+1)\equiv 0\pmod* {p_1}$. Thus, $\alpha^{i+j}\equiv -1\mod {p_1}$. But the smallest $k$ such that $\alpha^k\equiv -1\pmod* {p_1}$ is $k=z(p_1)=(p_1+1)/2$ which is odd. Hence, $i+j$ is an odd multiple of $z(p_1)$, therefore an odd number itself, which is a contradiction since $i\equiv j\pmod* 2$. Thus, this case does not appear.
This implies that $i\equiv j\pmod* {z(m)}$ if $U_i\equiv U_j\pmod* m$.\medskip
\hfil\break
\indent  Conversely, assume $i>j$ and $i\equiv j\pmod* {z(m)}$. We need to show that $U_i\equiv U_j\pmod* m$. Since $i\equiv j\pmod* {2^a}$, it follows that $i-j$ is even and hence $U_i-U_j=U_{(i-j)/2}V_{(i+j)/2}$. Since $2^{a-1}\mid (i-j)/2$, we get, by iteratively applying the formula $U_{2n}=U_nV_n$, that
$$
U_i-U_j=U_{(i-j)/2^{a}} V_{(i-j)/2^{a}} V_{(i-j)/2^{a-1}}\cdots V_{(i-j)/4} V_{(i+j)/2}.
$$
In the right--hand side we have $a$ factors from the $V$ sequence and each of
them is a multiple of $2$. Hence, $2^a\mid (U_i-U_j)$.  As for the divisibility by $p_1^{b_1}$, note that since $z(p_1^{b_1})\mid (i-j)$ and $i-j$ is even, it follows that
$$i-j=p_1^{b_1-1}(p_1+1)\ell,$$ for some positive integer $\ell$. Since $\alpha^{p_1^{b_1-1}(p_1+1)/2}\equiv -1\pmod* {p_1^{b_1}}$, it follows that $\alpha^{i-j}\equiv 1\pmod* {p_1^{b_1}}$. The same holds if we replace $\alpha$ by $\alpha^{-1}$. Thus,
$$
\alpha^i\equiv \alpha^j \alpha^{i-j}\equiv
\alpha^{j}\pmod*{p_1^{b_1}},
$$
and the same congruence holds if $\alpha$
is replaced by $\alpha^{-1}$. Subtracting
these two congruences we get $p_1^{b_1}\mid ((\alpha^i-\alpha^{-i})-(\alpha^j-\alpha^{-j}))$. Computing norms in ${\mathbb K}$ and using the fact that $p_1$ is inert, we get
$p_1^{2b_1}\mid N_{\mathbb K}((\alpha^i-\alpha^{-i})-(\alpha^j-\alpha^{-j}))$
and so $$p_1^{2b_1}\mid 32(U_i-U_j)^2.$$
Thus, $U_i\equiv U_j\pmod* {p_1^{b_1}}$. Hence, $U_i\equiv U_j\pmod* m$.
\end{proof}
\section{The end of the proof or why $5$ and not $37$?}
We need a few more results before we are prepared well enough to
establish Theorem \ref{main}.
\begin{lemma}
\label{BP}
For $n\ge 2^{24}\cdot 5^3$ the interval
$[5n/3,37n/19)$ contains a number of
the form
$2^a\cdot 5^{b}$ with $a\ge 1$ and $b\ge 0$.
\end{lemma}
\begin{proof}
It is enough to show that there exists an
strictly increasing sequence of integers
$\{m_i\}_{i=1}^{\infty}$ of the form
$m_i=2^{a_i+1}\cdot 5^{b_i}$ with
$a_1=23$ and $b_1=3$, $a_i,b_i\ge 0$, having
the property
that
$$1<\frac{m_{i+1}}{m_i}<\frac{111}{95}.$$
Since both
$2^7/5^3$ and $5^{10}/2^{23}$ 
are in $(1,111/95)$, the idea is to use
the substitutions $5^3\rightarrow 2^7$ and $2^{23}\rightarrow 5^{10}$ to
produce a strictly increasing sequence starting from $m_1$. Note that we
can at each stage make one of these substitutions as otherwise
we have reached a number dividing $2\cdot 2^{22}\cdot 5^{2}<m_1$,
a contradiction.
\end{proof}
\begin{corollary}
\label{geenwaarde}
Suppose that $m=2^a\cdot p^{b}$, $p>5$, $a,b\ge 1$.
If $m\ge \frac{37}{19}\cdot 2^{24}\cdot 5^{3}$, then
$m$ is not a discriminator value.
\end{corollary}
\begin{proof}
Suppose that $\mathcal{D}_1(n)=m$, then we must have
$$z(m)=2^a\cdot p^{b-1}(p+1)/(2k)\ge 19m/37\ge n,$$ that
is $m\ge 37n/19$. By Lemma \ref{BP} in the
interval $[5n/3,37n/19)$ there is an integer
of the form $m=2^c\cdot 5^{d}$ with $c\ge 1$. This integer
discriminates the first $n$ terms of the sequence and is smaller
than $m$. This contradicts the definition of the discriminator.
\end{proof}
Thus we see that in some sense there is an
abundance of the numbers of the form
$m=2^{a}\cdot 5^b$ that are in addition fairly
regularly distributed. Since they discriminate
the first $n$ terms provided that $m\ge 5n/3$, rather
than the weaker $m\ge 2np/(p+1)$ for $p>5$, they
remain as values, whereas numbers of the form
$m=2^{a}\cdot p^b$ with $p>5$ do not.
\begin{lemma}
\label{2a5b}
If $n>1,$ then ${\mathcal D}_1(n)=2^a\cdot 5^b$ for
some $a\ge 1$ and
$b\ge 0$.
\end{lemma}
\begin{proof}
By Lemma \ref{iseven} we have $a\ge 1$. If $m={\mathcal D}_1(n)\ne 2^a$
for some $a\ge 1$, then by Lemma \ref{restrict} iii it is
of the form $m=2^a\cdot p_1^{b_1}$ for some $p_1\equiv 5\pmod*{8}$.
Let assume for
the sake of contradiction that we have discriminators of the form $m=2^ap_1^{b_1}$ for some odd $p_1> 5$. Then $z(p_1^{b_1})=p_1^{b_1-1}(p_1+1)/2$. Let $A$ be minimal with $p_1^{b_1}<2^{A+8}$. Then $A\ge 2$ by our calculation
because we did not find any such $p_1$ on calculating ${\mathcal D}_1(n)$ for $n\le 2^{10}$ (cf. Section \ref{intro}).  Then
$$
2^{A+7}<p_1^{b_1-1}\frac{(p_1+1)}{2}.
$$
Consider the numbers
$$
2^{A+7},\quad  2^{A+1}\cdot 3\cdot 5^2,\quad 2^{A+1}\cdot 5^3,\quad 2^{A+8}.
$$
Assuming that $p_1>50$, $2^{A+8}>p_1^{b_1}$ and that $p_1^{b_1}+p_1^{b_1-1}>2^{A+8}$,
we obtain
$$
0<2^{A+8}-p_1^{b_1}<p_1^{b_1-1}<\frac{2^{A+8}}{50}.
$$
Hence, $p_1^{b_1}$ sits in the last $2\%$ of the interval ending at $2^{A+8}$. Since
$$
2^{A+8}-2^{A+1} 5^3=2^{A+1} 3=2^{A+8} (3/128)>2^{A+8}/50,
$$
it follows that $p_1^{b_1}>2^{A+1} 5^3$.  We now claim that $p_1^{b_1-1}(p_1+1)/2<2^{A+1} \cdot 3\cdot 5^2$. Indeed, for that it suffices that
$$
2^{A+8} \left(\frac{p_1+1}{2p_1}\right)<2^{A+1}\cdot 3\cdot 5^2,
$$
so $2p_1/(p_1+1)>128/75$, which is equivalent to $22p_1>128$, which is true for $p_1>50$.

Since $2^a\cdot p_1^{b_1}$ discriminates the first
$2^a\cdot p_1^{b_1-1}(p_1+1)/2$ terms of the sequence (but not more)
and the same integers are discriminated by the smaller number
$2^{a+A+1}\cdot 5^3$, the number $2^a\cdot p_1^{b_1}$ is not a
discriminator value.

%The conclusion now is that for any positive integer $a$
%$$
%[2^{a+A+7}+1,~2^{a-1}p_1^{b_1-1}(p_1+1)]\subset  [2^{a+A+7},~2^{a+A+1}\cdot %3\cdot 5^2]
%$$
%so for the positive integers $n$ in the left--interval, $2^{a+A+1}\cdot 5^3$ %is a better discriminator than $2^ap_1^{b_1}$.
So, it remains to check primes $p_1<50$. Since $p_1\equiv 5\pmod* 8$, we just need to check $p_1\in \{13,29,37\}$. Fortunately, $13$ is a Wiefrich prime for $\alpha$ in that
$$
(3+2{\sqrt{2}})^7\equiv -1\pmod*{13^2},
$$
so we cannot use powers $13^{b_1}$ with $b_1>1$, while for $b_1=1$ the interval $[z(p_1),p_1]=[7,13]$ contains $8$ which is a power of $2$. For $29$, we have that $z(29)=5$ (instead of $(29+1)/2$), so
$29$ is not good either.

It remains to deal with $p=37$. We will show
that for $k_i=2\cdot 37^i$ and $1\le i\le 5$,
there is a power of the form $2^{e_i}<k_i$ that
discriminates the same terms of the sequence
as $k_i$ does, thus showing that $k_i$ cannot
be a discriminator. By the same token,
any potential value $2^{\alpha}\cdot 37^i$, $1\le i\le 6$,
is outdone by $2^{\alpha+e_i}$. Any remaining
value of the form $2^{\alpha}\cdot 37^i$ has
$i\ge 6$ and $\alpha\ge 1$ and cannot be a value by
Corollary \ref{geenwaarde}.
\par The numbers $2^{e_i}$ we are looking for
must satisfy $$2\cdot 37^{i-1}\cdot 19\le
2^{e_i}< 2\cdot 37^i\text{~for~}1\le i\le 5.$$
(Recall that the number $2\cdot 37^i$
discriminates the first $2\cdot 37^{i-1}\cdot 19$
terms of the sequence and not more terms.)
Note that thee numbers $e_i$ are unique if
they exist. Some
simple computer algebra computations yield
$e_1=6,e_2=11,e_3=16,e_4=21$ and $e_5=27$.
\end{proof}
\begin{lemma}
\label{twocases}
We say that $m$ discriminates $U_0,\ldots,U_{n-1}$
if these integers are
pairwise distinct modulo $m$.
\begin{itemize}
\item[{\rm i)}] The integer $m=2^a$ discriminates
$U_0,\ldots,U_{n-1}$ if and only if $m\ge n$.
\item[{\rm ii)}] The integer $m=2^a\cdot 5^b$
with $a,b\ge 1$ discriminates
$U_0,\ldots,U_{n-1}$ if and only if
$m\ge 5n/3$.
\end{itemize}
\end{lemma}
\begin{proof}
Case i follows from Lemma \ref{poweroftwo}.
Now suppose that $a,b\ge 1$.
By Lemma \ref{lem:last} the integer $m$ discriminates
$U_0,\ldots,U_{z(m)-1}$, but not
$U_0,\ldots,U_{z(m)}$. It follows that
$m$ discriminates $U_0,\ldots,U_{n-1}$ iff $n\le z(m)$.
As it is easily seen that $z(m)=3m/5$, the result follows.
\end{proof}
At long last we are ready to prove Theorem \ref{main}.
\begin{proof}[Proof of Theorem \ref{main}]
As the statement is correct for $n=1$, we may assume
that $n>1$. By Lemma \ref{2a5b} it then follows that
either $m=2^a$ for some $a\ge 1$ or
$m=2^a\cdot 5^b$ with $a,b\ge 1$.
On invoking Lemma \ref{twocases} we infer
that the first assertion holds true.
\par It remains to determine the image of
the discriminator ${\mathcal D}_1$.
Let us suppose that $m=2^a\cdot 5^b$ with $a,b\ge 1$
occurs as value. Let $\alpha$ be the unique integer
such that $2^{\alpha}<2^a\cdot 5^b<2^{\alpha+1}$.
By Lemma \ref{twocases} it now follows that we must have
$z(m)>2^{\alpha}$, that is
$2^a\cdot 5^{b-1}\cdot3>2^{\alpha}$. It follows that $m$ occurs
as value iff
\begin{equation}
\label{ongelijk}
\frac{5}{3}\cdot 2^{\alpha}<2^a\cdot 5^b<2^{\alpha+1}.
\end{equation}
(Indeed, under these
conditions we have ${\mathcal D}_1(n)=2^a\cdot 5^b$ for
$n\in [2^a+1,2^a\cdot 5^{b-1}\cdot3]$.) Inequality \eqref{ongelijk}
can be rewritten as $5/6<2^{a-\alpha-1}<1$ and, after
taking logarithms, is seen to have
a solution iff $b\in {\mathcal M}$.
If it has a solution, then we must have
$\alpha-a=\lfloor b\log 5/\log 2\rfloor$. In
particular for each $a\ge 1$ and $b\in {\mathcal M}$,
the number $2^a\cdot 5^b$ occurs as value.
\end{proof}

\section{General $k$}
\subsection{Introduction}
\label{generalkintro}
What is happening for $k>1$? It turns out that the situation is quite different. 
\par For $k=2$ we have the following result.
\begin{theorem}
\label{tweee}
Let $e\ge 0$ be the smallest integer such that
$2^e\ge n$ and $f\ge 1$ the smallest  integer such that
$3\cdot 2^f\ge n$. Then
${\mathcal D}_2(n)=\min\{2^e,3\cdot 2^f\}$.
\end{theorem}
\begin{proof}
We have that if
$z(m)=m$, then $m|3\cdot 2^a$ for some $a\ge 0$.
For the other integers $m$ we have $z(m)\le 3m/5$ (actually
even $z(m)\le 7m/13$).
It follows that if $m$ discriminates the first $n$ values
of the sequence $U(2)$, then we must have
$m\ge 5n/3$.
It is easy to check that for every $n\ge 2$
there is a power of two or a number of the form
$3\cdot 2^a$ in the interval $[n,5n/3)$.
As ${\mathcal D}_2(1)=1$ we are done.
\end{proof}

For the
convenience of the reader we recall the
theorem from the introduction which deals
with the case $k>2$.
\begin{theorem}
\label{thm:2}
Put
$${\mathcal A}_k=\begin{cases}
\{m~{\text{\rm odd}}:\text{\rm if~}p\mid m,~{\text{\rm then}}~p\mid k\}\text{~if~}
k\not\equiv 6\pmod*{9};\cr
\{m~{\text{\rm odd}},~9\nmid m:\text{\rm if~}p\mid m,~{\text{\rm then}}~p\mid k\}\text{~if~}
k\equiv 6\pmod*{9},
\end{cases}
$$
and $${\mathcal B}_k
=\begin{cases}
\{m~{\text{\rm even}}:\text{\rm if~}p\mid m,~{\text{\rm then}}~p\mid k(k+1)\}
\text{~if~}
k\not\equiv 2\pmod*{9};\cr
\{m~{\text{\rm even},~9\nmid m}:\text{\rm if~}p\mid m,~{\text{\rm then}}~p\mid k(k+1)\}
\text{~if~}
k\equiv 2\pmod*{9}.
\end{cases}
$$
Let $k>2$. We have
\begin{equation}
\label{dkinequal2}
{\mathcal D}_{k}(n)\le \min\{m\ge n:
m\in {\mathcal A}_{k}\cup {\mathcal B}_{k}\},
\end{equation}
with equality if the interval $[n,3n/2)$ contains an integer 
$m\in {\mathcal A}_{k}\cup {\mathcal B}_{k}$.
There are at most 
finitely many $n$ for which in \eqref{dkinequal2} strict inequality holds.
Furthermore, we have 
\begin{equation}
\label{Dkn=n}
{\mathcal D}_{k}(n)=n 
\iff n\in {\mathcal A}_{k}\cup {\mathcal B}_{k}.
\end{equation}
\end{theorem}

In our proof of this result the rank of appearance
plays a crucial role. Its most 
important properties are summarized in Lemma \ref{appearance}.
\subsection{The index of appearance}
\subsubsection{The case where $p\mid k(k+1)$}
The index of appearance for
primes $p$ dividing $k(k+1)$ is determined in
Lemma \ref{speciaal} for $p>2$.
By Lemma \ref{poweroftwo} we have
$z(2^b)=2^b$. In general $z(p^b)=p^b$ for
these primes, but for a prime which we call {\it special} a
complication can arise giving rise to
$z(p^b)\mid p^{b-1}$ for $b\ge 2$.
\begin{defi}
A prime $p$ is said to be special if
$p|k(k+1)$ and $p^2|U_p$.
\end{defi}
The special feature of a special prime
$p$ is that $p^b$ with $b\ge 2$ cannot
divide a discriminator value.
Recall that $z(p^a)=p^{\max\{a-c,0\}} z(p)$, where $c=\nu_p(U_{z(p)})$ by Lemma \ref{lem:flauw}.
\begin{lemma}
\label{drienogmaals}
Let $p\ge 3$ be an odd prime.
If $z(p^{b})|p^{b-1}$, then $m=p^bm_1$
with $p\nmid m_1$ is not a discriminator
value.
\end{lemma}
\begin{proof}
Taking $i=0$ and $j=p^{b-1}z(m_1)$ we have
$U_i\equiv U_j\equiv 0 \pmod*{m}.$
It follows that $n\le p^{b-1} z(m_1)\le m/p$ so any power of 2 in $[m/3,m)$ (and such a
power exists) is a better discriminator than $m$.
\end{proof}
By Lemma \ref{speciaal} only 3 can be special.
\subsubsection{The case where $p\nmid k(k+1)$}
Let us now look at odd prime numbers $p$ such that $p\nmid k(k+1)$. These come in two  flavors according to the sign of 
\begin{equation}
\label{ep}
e_p=\Big(\frac{k(k+1)}{p}\Big).
\end{equation}
Suppose that $e_p=1$. Then
either
$$
\left(\frac{k}{p}\right)=\left(\frac{k+1}{p}\right)=1\quad {\text{\rm or}}\quad \left(\frac{k}{p}\right)=\left(\frac{k+1}{p}\right)=-1.
$$
In the first case,
\begin{eqnarray*}
\beta^p & = & ({\sqrt{k+1}}+{\sqrt{k}})^p\equiv {\sqrt{k+1}} (k+1)^{(p-1)/2}+{\sqrt{k}} k^{(p-1)/2}\\
& \equiv &  {\sqrt{k+1}}+{\sqrt{k}}\equiv \beta\pmod*{p}.
\end{eqnarray*}
In the second case, a similar calculation shows that $\beta^p\equiv -\beta$. Thus, $\beta^{p-1}\equiv \pm 1\pmod* p$ and since $\alpha=\beta^2$, we get that $\alpha^{(p-1)/2}=\beta^{p-1}\equiv \pm 1\pmod* p$. Since the
last congruence implies that $\alpha^{-(p-1)/2}\equiv \pm 1\pmod* p$ we obtain
on subtracting these two congruences that
$p\mid U_{(p-1)/2}$. Thus, $z(p)\mid (p-1)/2$.
In case $e_p=-1$, a similar calculation shows that $\beta^p\equiv \pm \beta^{-1}\pmod* p$, so $\beta^{p+1}\equiv \pm 1\pmod p$. This shows that $\alpha^{(p+1)/2}\equiv \pm 1\pmod* p$, which leads to $z(p)\mid (p+1)/2$.  
There is one more observation which is useful here. Assume that $e_p=-1$, which implies
that $z(p)\mid (p+1)/2$. Suppose that $p\equiv 3\pmod* 4$. Then $(p+1)/2$ is even. Assume further that
$$
\left(\frac{k+1}{p}\right)=1\quad{\text{\rm ~~and ~~}}\quad\left(\frac{k}{p}\right)=-1.
$$
In this case, by the above arguments, we have that $\beta^{p}\equiv \beta^{-1}\pmod* p$, so $\beta^{p+1}\equiv 1\pmod* p$. This gives $\alpha^{(p+1)/2}\equiv 1\pmod* p$. Since $(p+1)/2$ is even we conclude that
$$
p\mid(\alpha^{(p+1)/4}-1)(\alpha^{(p+1)/4}+1).
$$
Since $p$ is inert in ${\mathbb K}$, we get that $\alpha^{(p+1)/4}\equiv \pm 1\pmod* p$, which later leads to $p\mid U_{(p+1)/4}$, Hence, $z(p)\mid (p+1)/4$ in this case.

\subsubsection{General $m$}
\begin{lemma}
\label{appearance} Let $k\ge 1$.
We have $z(m)=m$ if and only if
$$
\begin{cases}
m\in {\mathcal P}(k(k+1)),~9\nmid m;\cr
m\in {\mathcal P}(k(k+1)),~9\mid m,~
\text{and~}3{~is~not~special.}
\end{cases}
$$
For the remaining integers $m$ we have
$$z(m)\le \alpha_k m,$$
with
\begin{equation}
\label{alphadef}
\alpha_k:=\lim\sup_{m\rightarrow \infty}\Big\{\frac{z_k(m)}{m}:z_k(m)<m\Big\}.
\end{equation}
One has
\begin{equation}
\label{alpha}
\alpha_k=\lim\sup_{p\rightarrow \infty}\Big\{\frac{z_k(p)}{p}:z_k(p)<p\Big\}.
\end{equation}
Furthermore, we have 
$\alpha_k=2/3$ if $k\equiv 1\pmod*{3}$ and
$\alpha_k\le 3/5$ otherwise.
\end{lemma}
\begin{corollary}
\label{zongelijk}
We have $z(m)\le m$.
\end{corollary}
\begin{corollary}
\label{cor:appearance}
We have
\begin{equation*}
{\mathcal A}_k=\{m~{\text{\rm odd}}:z(m)=m\text{\rm ~and~}
m\in {\mathcal P}(k)\},
\end{equation*}
and
\begin{equation*}
{\mathcal B}_k=\{m~\text{\rm even}:z(m)=m\}.
\end{equation*}
\end{corollary}
\begin{proof}[Proof of Lemma \ref{appearance}]
By the above discussion if $p\nmid k(k+1)$, then
$z(p^b)<p^b$. Thus if $z(m)=m$, then 
$m\in {\mathcal P}(k(k+1))$. The first assertion
now follows by Lemma 
\ref{poweroftwo} (which shows that
$z(2^b)=2^b$) and Lemma \ref{speciaal} and
the observation that if 
$m=\prod_{i=1}^s p_i^{b_i}$
is the factorization of $m$ with
$z(p_i^{b_i})=p_i^{b_i}$, then
$$z(m)=\text{\rm lcm}(z(p_1^{b_1}),\ldots,z(p_s^{b_s}))=
\prod_{i=1}^s p_i^{b_i}=m.$$
\par If $m=\prod_{i=1}^s p_i^{b_i}$ is the
factorization of {\it any} integer, then
$$\frac{z(m)}{m}\le 
\prod_{i=1}^s \frac{z(p_i^{b_i})}{p_i^{b_i}}
\le \prod_{i=1}^s \frac{z(p_i)}{p_i}.$$
{}From these inequalities we infer the
truth of \eqref{alpha}. The proof is 
concluded on noting that 
$$z(3)=
\begin{cases}
2 & \text{if~}k\equiv 1\pmod*{3};\cr
3 & \text{otherwise,}
\end{cases}
$$
and
that $(p+1)/2p$ is a decreasing
function of $p$.
\end{proof}
It is easy to see that if there is a prime
$p$ with $z(p)=(p+1)/2$, then
$$\alpha_k=\frac{q+1}{2q},$$
where $q$ is the smallest prime such 
that $z(q)=(q+1)/2$.

\subsection{The congruence $U_i(k)\equiv U_j(k)\pmod* m$}
In this subsection we study the congruence $U_i(k)\equiv U_j(k)\pmod* m$. As we said before, it suffices to study it modulo prime powers. For powers of $2$, this has been done at the beginning of Section \ref{sec:2}.
So, we deal with prime powers $p^b$.
Recall that the discriminant
$\Delta(k)$ equals $16k(k+1)$. It turns
out that primes $p$ dividing $\Delta(k)$ are easier to understand
than the others.
From now on, we eliminate the index $k$ from $U_n(k),~\alpha(k),~\Delta(k)$ and so on. We treat the case when
$p\mid k(k+1)$. In case $m$ is even, there are two subcases, one easy and one harder, according to whether $p\mid k$ or $p\mid (k+1)$.

\begin{lemma}
\label{elf}
Assume $p\mid k$ is odd. Then $U_i\equiv U_j\pmod* {p^b}$ if and only if $i\equiv 
j\pmod* {z(p^b)}$.
\end{lemma}

\begin{proof}
We prove the only if assertion. We let $a$ be such that $p^a\| k$. We put $k(k+1)=du^2$, and let ${\mathbb K}={\mathbb Q}[{\sqrt{d}}]$. We let $\pi$ be any prime ideal diving $p$ and let $e$ be such that $\pi^e\| p$. For example, $e=2$ if $p\mid d$. Let $\lambda$ be the residue class of  the number $U_i$ modulo $p^b$. Then $U_i\equiv \lambda\pmod* {p^b}$ implies that
$$
\alpha^i-\alpha^{-i}-4{\sqrt{k(k+1)}\lambda}\equiv 0\pmod*{\pi^{eb+ae/2}}.
$$
The same holds for $\alpha^i$ replaced by $\alpha^j$. Hence, these numbers both satisfy the quadratic congruence
$$
x^2-4{\sqrt{k(k+1)}}\lambda x-1=0\pmod* {\pi^{eb+ae/2}}.
$$
Taking their difference we get
\begin{equation}
\label{eq:cong}
(\alpha^i-\alpha^j)(\alpha^i+\alpha^j -4{\sqrt{k(k+1)}})\lambda\equiv 0\pmod* {\pi^{be+ae/2}}.
\end{equation}
In case $p\mid k$, we have that $\alpha=2k+1+2{\sqrt{k(k+1)}}\equiv 1\pmod* {\pi}$. Thus, the  second factor above is congruent to $2\pmod*{\pi^{ae/2}}$. In particular, $\pi$ is coprime to that factor. Thus,
$$
\alpha^i\equiv \alpha^j\pmod*{\pi^{be+ae/2}}.
$$
This leads to $\alpha^{i-j}\equiv 1\pmod* {\pi^{be+ae/2}}$. Changing $\alpha$ to $\alpha^{-1}$ and taking the difference of the above expressions we $\alpha^{i-j}-\alpha^{j-i}\equiv 0\pmod* {\pi^{be+ae/2}}$. Thus,
$$
2{\sqrt{k(k+1)}} U_{i-j}\equiv 0\pmod* {\pi^{be+ae/2}}.
$$
Clearly, the exponent of $\pi$ in $2{\sqrt{k(k+1)}}$ is exactly $ae/2$. Thus, $\pi^{eb}\mid U_{i-j}$. Since this is true for all prime power ideals $\pi^e$ dividing $p$, we get that $p^b\mid U_{i-j}$. Thus,
$i-j\equiv 0\pmod* {z(p^b)}$.

For the if assertion, assume that $i\equiv j\pmod*{z(p^b)}$. Then the congruence
$U_{i-j}\equiv 0 \pmod*{p^b}$ holds which implies
$\alpha^{i-j}=\alpha^{-(i-j)} \pmod* {\pi^{eb+ae/2}}$. In turn this gives $\alpha^{2(i-j)}-1\equiv 0 \pmod*{\pi^{eb+ae/2}}$ so
$(\alpha^{i-j}-1)(\alpha^{i-j}+1)\equiv 0 
\pmod* {\pi^{eb+ae/2}}$. Since $\alpha\equiv 1 
\pmod*{\pi}$,  the factor  $\alpha^{i-j}+1$ is congruent to $2\pmod*{\pi}$, so coprime to $\pi$. So $\alpha^{i-j}\equiv 1 \pmod* {\pi^{eb+ae/2}}$, giving
$\alpha^i-\alpha^j\equiv 0 \pmod* {\pi^{eb+ae/2}}$. Since $\alpha$ is a unit we also get $\alpha^{-i}-\alpha^{-j}\equiv 
0\pmod*{\pi^{eb+ae/2}}$.
Taking the difference of the last two congruences, we get
$$2\sqrt{k(k+1)} (U_i-U_j)\equiv 0 \pmod* {\pi^{eb+ae/2}}.$$ 
Simplifying the square-root which contributes a power $\pi^{ae/2}$ to the left--hand side of the above congruence, we get 
$$U_i\equiv U_j \pmod* {\pi^{eb}},$$ and since this is true for all $\pi\mid p$, 
we get that $U_i\equiv U_j\pmod*{p^b}$.
\end{proof}

Now we treat the more delicate case $p\mid (k+1)$.
The following lemma is the
analogue of Lemma \ref{elf}.
\begin{lemma}
\label{tweedeelgevallen}
Assume that $p$ is odd and $p\mid (k+1)$. Then $U_i\equiv U_j\pmod* {p^b}$ is equivalent to one of the following:
\begin{itemize}
\item[{\rm i)}] If $i\equiv j\pmod* 2$, then $i\equiv j\pmod* {z(p^b)}$.

\item[{\rm ii)}] If $i\not\equiv j\pmod* 2$, then $i+j\equiv 0\pmod* {z(p^b)}$.
\end{itemize}
\end{lemma}

\begin{proof}
The proof is similar to the previous lemma. Let $p^a\mid (k+1)$ and let $\pi$ be some prime ideal in ${\mathbb K}$ such that $\pi^e\mid p$. Then $$\alpha=2k+1+2{\sqrt{k(k+1)}}\equiv -1\pmod*{\pi^{ae/2}}.$$ Let again
$\lambda$ be the value of $U_i\pmod*{p^b}$. The same argument as before leads us to the congruence \eqref{eq:cong}. The first factor is congruent to
$$
(-1)^i-(-1)^j\pmod*{\pi^{ae/2}}.
$$
The second one is congruent to 
$(-1)^i+(-1)^j\pmod*{\pi^{ae/2}}$. Thus, $\pi$ never divides both factors, and $\pi^{ae/2}$ divides $\alpha^i-\alpha^j$ in case $i\equiv j\pmod*{2}$, and it divides $\alpha^i+\alpha^j-4{\sqrt{k(k+1)}} \lambda$
in case $i\not\equiv j\pmod*{2}$. 

In case $i\equiv j\pmod* 2$, we have $\alpha^i\equiv \alpha^j\pmod* {\pi^{be+ae/2}}$. Thus, $\alpha^{i-j}\equiv 1\pmod* {\pi^{be+ae/2}}$. Arguing as
in the proof of the preceding lemma yields $U_{i-j}\equiv 0\pmod*{p^b}$
and hence $i\equiv j\pmod* {z(p^b)}$.

Assume now that $i\not\equiv j\pmod* {2}$.
Multiply both sides of the congruence
$$
\alpha^i+\alpha^j-4{\sqrt{k(k+1)}} \lambda\equiv 0\pmod*{\pi^{ae/2+be}}.
$$
by $\alpha^j$ and rewrite it as
$$
\alpha^{i+j}\equiv -\alpha^{2j}+4\alpha^j {\sqrt{k(k+1)}} \lambda\pmod*{\pi^{ae+be}}.
$$
Since $\pi^{ae/2}\mid 4{\sqrt{k(k+1)}} \alpha^j$, it follows that the value of the right--hand side is determined by 
$\lambda\pmod* {\pi^{be}}$, which is $(\alpha^j-\alpha^{-j})/(4{\sqrt{k(k+1)}})$.
Thus,
$$
-\alpha^{2j}+4\alpha^j {\sqrt{k(k+1)}} \lambda\equiv -\alpha^{2j}+\alpha^j (\alpha^{j}-\alpha^{-j})\equiv 
-1\pmod* {\pi^{be+ae/2}}.
$$  
So we get that $\alpha^{i+j}\equiv -1\pmod* {\pi^{be+ae/2}}$. The same holds with $\alpha$ replaced by $\alpha^{-1}$. Subtracting both congruences we get that
$$
\pi^{be+ae/2}\mid (\alpha^{i+j}-\alpha^{-i-j})=4{\sqrt{k(k+1)}} U_{i+j},
$$
leading to $(\pi^e)^b\mid U_{i+j},$
and thus to $z(p^b)\mid (i+j).$

We now have to do the if parts. They are pretty similar to the previous analysis.
We start with $i\equiv j \pmod* 2$. Then $i-j\equiv 0 \pmod* {z(p^b)}$, so $U_{i-j}\equiv 0\pmod* {p^b}$. This gives as in the previous case
$\alpha^{i-j}\equiv \alpha^{-(i-j)} \pmod* {\pi^{eb+ae/2}}$, so $\alpha^{2(i-j)}\equiv 1 \pmod*{\pi^{eb+ae/2}}$. Thus, 
$(\alpha^{i-j}-1)(\alpha^{i-j}+1) \equiv 0 
\pmod* {\pi^{be+ae/2}}$. Since $i-j$ is even, $\alpha^{i-j}\equiv (-1)^{i-j}\pmod* \pi \equiv 1 
\pmod*{\pi}$, 
so the second factor is congruent to $2 \pmod*{\pi}$, so it is coprime to $\pi$. So,
$\alpha^{i-j}-1\equiv 0\pmod* {\pi^{be+ae/2}}$. Now the argument continues as in the last part of the proof of the preceding lemma to get to the conclusion that
$U_i\equiv U_j\pmod*{p^b}$.

A similar argument works when $i\not\equiv j\pmod* 2$. With the same argument we get from $i+j\equiv 0\pmod* {z(p^b)}$ to the relation $U_{i+j}\equiv 0\pmod* {p^b}$, which
on its turn leads to $(\alpha^{i+j}-1)(\alpha^{i+j}+1)\equiv \pmod* {\pi^{be+ae/2}}$. Since $i+j$ is odd, the  factor $\alpha^{i+j}-1$ is congruent to is $-2 \pmod*{\pi}$,
so it is  coprime to $\pi$. So, $\alpha^{i+j}+1\equiv 
0\pmod* {\pi^{eb+ae/2}}$ and multiplying with a suitable power of $\alpha$ and rearranging we get 
$\alpha^i\equiv -\alpha^{-j} \pmod* {\pi^{be+ae/2}}$, and also 
$\alpha^{-i}\equiv -\alpha^j \pmod* {\pi^{be+ae/2}}$. Taking the difference of these last two congruences, we get
$\alpha^i-\alpha^{-i}-\alpha^j+\alpha^{-j}\equiv 0 \pmod* {\pi^{be+ae/2}}$, which is  
$2\sqrt{k(k+1)}(U_i-U_j)\equiv 0 \pmod* {\pi^{be+ae/2}}$. Simplifying $2\sqrt{k(k+1)}$, we get that $\pi^{be}$ divides $U_i-U_j$, and since $\pi$ is an arbitrary prime ideal of $p$, we conclude that $U_i\equiv U_j\pmod* {p^b}$.
\end{proof}
\begin{definition}
We write ${\mathcal P}(r)$ for the set of 
positive integers composed only of prime
factors dividing $r$.
\end{definition}
\begin{lemma}
\label{lem:iui}
We have 
\begin{equation}
\label{iui}
i\equiv j\pmod*{m} 
\iff U_i\equiv U_j\pmod*{m},
\end{equation}
precisely when 
$$m\in {\mathcal A}_k\cup {\mathcal B}_k.$$
\end{lemma}
\begin{proof} Since $0=U_0(k)\equiv U_{z(m)}(k)\pmod* m$, 
we must have $z(m)\ge m$. As $z(m)\le m$ by Corollary \ref{ongelijk} 
it follows that $z(m)=m$.
\par First subcase: $m$ is odd.\hfil\break
Since $z(m)=m$ all prime divisors of $m$ 
must divide
$k(k+1)$. Now suppose that $m$ has an odd prime
divisor $p$ dividing $k+1$. Thus 
$m=p^am_1$ with $m_1$ coprime to $p$ and odd.
Note that $z(p^a)=p^a$.
Consider $i=(p^{a}-1)m_1/2$ and $j=(p^{a}+1)m_1/2$. Then $i\not\equiv j\pmod* 2$ and $p^{a}\mid (i+j)$. Thus, $U_i\equiv U_j\pmod*{p^{a}}$
by Lemma \ref{tweedeelgevallen}.
Since
$m_1\mid i$ and $m_1\mid j$ and
$m_1$ is composed of primes dividing
$\Delta(k)=16k(k+1)$, it follows that $U_i\equiv 
U_j\equiv 0\pmod*{m_1}$ and hence
we have $U_i\equiv U_j\pmod*{m}$ with
$m\nmid (j-i)$. It follows that \eqref{iui}
is not satisfied. Thus we conclude that
if an odd integer $m$ is to satisfy \eqref{iui}
it has to be in ${\mathcal P}(k)$. For such
an integer, by Lemma \ref{elf} and the 
Chinese remainder theorem, \eqref{iui} is
always satisfied.
It follows that the solution set of odd
$m$ satisfying \eqref{iui} is
$\{m~{\text{\rm odd}}:z(m)=m\text{\rm ~and~}
m\in {\mathcal P}(k)\},$ which
by Corollary \ref{cor:appearance} equals ${\mathcal A}_k$.
\par Second subcase: $m$ is even.\hfil\break
Both the left and the right side of 
\eqref{iui} imply
that $i\equiv j\pmod*{2}$.
On
applying Lemmas \ref{elf} and \ref{tweedeelgevallen} and the
Chinese remainder theorem we see that
in this case the solution set is 
$\{m~\text{\rm even}:z(m)=m\},$ which
by Corollary \ref{cor:appearance} equals ${\mathcal B}_k$.
\end{proof}

\subsection{A Diophantine interlude}
The prime 3 sometimes being special
leads us to solve a very
easy Diophantine problem (left to
the reader).
\begin{lemma}
\label{dioflauw}
If $k>2$, then $k(k+1)$ has an odd prime factor
that is not special.
\end{lemma}
\begin{proof}
If $k(k+1)$ only has an odd prime factor that is special, then
it must be $3$ and $k\equiv 2,6\pmod*{9}$. It follows that for
such a $k$ there are $a,b$ for which the Diophantine equation
\begin{equation}
\label{Diophantine}
k(k+1)=2^a\cdot 3^b,
\end{equation}
has a solution. However, this is easily shown to 
be impossible for $k>2$.
\end{proof}
It is slightly more challenging to
find {\tt all} solutions $k\ge 1$ of \eqref{Diophantine}.
In that case one is led
to the Diophantine equation
$$2^a-3^b\equiv \pm 1,$$
which was already solved centuries
ago by Levi ben Gerson (alias Leo
Hebraeus), who lived in Spain from
1288 to 1344, cf. Ribenboim \cite[p. 5]{Rib}.
It has the solutions $(a,b)=(1,0),(0,1),(2,1)$
and $(a,b)=(3,2)$, corresponding to,
respectively, $k=1,2,3$ and $k=8$.

\subsection{Bertrand's Postulate for S-units}
Before we embark on the proof of our main result we
make a small excursion in Diophantine approximation.
\begin{lemma}
\label{eventuallythere}
Let $\alpha>1$ be a real number and 
$p$ be an arbitrary odd prime. Then there 
exists a real number $x(\alpha)$ such
that for every $n\ge x(\alpha)$ the interval $[n,n\alpha)$ contains
an even integer of the form
$2^a\cdot p^b$.
\end{lemma}
\begin{proof}
Along the lines of the proof 
of Lemma \ref{BP}. If $\beta$
is irrational, then the sequence of integers
$\{m\beta\}_{m=1}^{\infty}$ is uniformly distributed.
This allows one to find quotients
$2^c/p^d$ and $p^r/2^s$ that are in the interval
$(1,\alpha)$. Then proceed as in the proof
of Lemma \ref{BP}.
\end{proof}
The result also holds for S-units of the 
form $\prod_{i=1}^s p_i^{b_i}$ with $p_1<\ldots < p_s$ primes
and $s\ge 2$. 
\par In \cite{CLM} we consider
the {Bertrand's Postulate for S-units in 
greater detail.

\subsection{Proof of the main result for
general $k$}
Finally we are in the position to prove our main
result for $k>1$.
\begin{proof}[Proof of Theorem \ref{thm:2}] Let $k>2$.
\hfil\break
\indent First case: $m\in {\mathcal A}_k\cup {\mathcal B}_k.$ (Note that 
$z(m)=m$ for these $m$.)\hfil\break
By Lemma \ref{lem:iui} we infer that the inequality \eqref{dkinequal2}
holds true and moreover the equivalence 
\eqref{Dkn=n}. The ``$\Leftarrow$" implication
in \eqref{Dkn=n} yields
${\mathcal A}_k\cup {\mathcal B}_k
\subseteq {\mathcal D}_k$.
\par Second case: $z(m)=m$ and $m\not \in
{\mathcal A}_k\cup {\mathcal B}_k$.\hfil\break
In this case $m$ has a 
odd prime divisor $p$ 
that also divides $k+1$. 
Now write $m=p^a\cdot m_1$ with $p\nmid m_1$ and $m_1$ odd. Note
that $z(p^a)=p^a$.
Consider $i=(p^a-1)m_1/2$ and $j=(p^a+1)m_1/2$. Then $i\not\equiv j\pmod* 2$ and 
$p^{a}\mid (i+j)$. Thus, $U_i\equiv U_j\pmod* {p^{a}}$
by Lemma \ref{tweedeelgevallen}.
Since
$m_1\mid i$ and $m_1\mid j$ and
$m_1$ is composed of primes dividing
$\Delta(k)$, it follows that $U_i\equiv U_j\equiv 0\pmod* {m_1}$. This shows that
if $m$ discriminates the 
numbers $U_0(k),\ldots,U_{n-1}(k),$ then
$$
n\le \left(\frac{p^{a}+1}{2}\right)m_1.
$$
The interval $[(p^a+1)/2,p^{a})$ contains a power of $2$, say $2^b$. Then 
$2^bm$ is a better discriminator than $p^{a}m_1=m$. 
Thus if $z(m)=m$ and $m\not \in
{\mathcal A}_k\cup {\mathcal B}_k,$ then $m$ is not a discriminator value.
\par Third case: $z(m)<m$.\hfil\break
Here it follows by Lemma \ref{appearance}
that $z(m)\le \alpha_k m\le 2m/3$. In order
for $m$ to discriminate the first $n$ terms
we must have $n\le z(m)\le 2m/3$, that is
$m\ge 3n/2$. Now if in the interval $[n,3n/2)$  
there is an element from  ${\mathcal A}_k\cup {\mathcal B}_k,$
this will discriminate the first $n$ terms too and
is a better discriminator than $m$. Thus in this
case in \eqref{dkinequal2} we have equality.
\par Since by assumption $k>2$, by 
Lemma \ref{dioflauw} there exists
a non-special
odd prime $p$ dividing $k(k+1)$ 
and hence if $a,b\ge 0,$ then $2^{1+a}\cdot p^b
\in {\mathcal A}_k\cup {\mathcal B}_k.$ It now
follows by Lemma \ref{eventuallythere} that for every $n$ large enough the interval $[n,3n/2)$  
contains an element from  ${\mathcal A}_k\cup {\mathcal B}_k$ and so 
there are at most 
finitely many $n$ for which in \eqref{dkinequal2} strict inequality holds.
\end{proof}
\subsection{The set ${\mathcal F}_k$}
As was remarked in the introduction a consequence of Theorems \ref{twee}
and \ref{main2} is that
for $k>1$ there is a finite set ${\mathcal F}_k$ such that
$$
{\mathcal D}_{k}=
{\mathcal A}_k\cup
{\mathcal B}_k\cup {\mathcal F}_k.
$$
The set ${\mathcal F}_k$ 
is not a figment
of our proof of this result, as the
following result shows.
\begin{lemma}
There are infinitely many $k$ for the 
finite set ${\mathcal F}_k$ is non-empty.
It can have a cardinality larger than
any given bound.
\end{lemma}
\begin{proof}
Let $N$ be large and $k\equiv 1\pmod* {N!}$. Then $U(k)\pmod* m$ is the same as $U(1)\pmod* m$ for all
$m\le N$. In particular, if $N>2\cdot 5^{m_s}$, where $m_s$ is the
$s^\text{th}$ element of the set ${\mathcal M}$, then certainly ${\mathcal D}_1\cap [1,N]$ will contain the numbers $2\cdot 5^{m_i}$ for $i=1,\ldots,s$, and $5\nmid k(k+1)$ (in fact,
$k\equiv 1\pmod*{5}$, so $5\nmid k(k+1)$), therefore all such numbers
are in  the  set ${\mathcal F}_k$ for such values of $k$.
\end{proof}
Thus it is illusory
to want to describe ${\mathcal F}_k$ 
completely for every $k\ge 1$. Nevertheless, 
in part II \cite{CLM} we will explore
how far we can get in this respect.
\section{Analogy with the polynomial
discriminator}
In our situation for $k\ge 1$ on the one hand there
are enough $m$ with $z(m)=m$ and 
${\mathcal D}_k(m)=m$, on the other hand for the
remaining $m$ either $z(m)=m$ and $m$ is not a discriminator
value or we have $z(m)\le \alpha_km$ 
with $\alpha_k<1,$ a constant not depending on 
$m$. Thus the distribution of
$\{z(m)/m:m\ge 1\}$ shows a gap directly below 1 (namely
$(\alpha_k,1)$).
\par For polynomial discriminators the analogue of $z(p)$ is $V(p)$, the number of values
assumed by the polynomial modulo $p$. If on
the one hand there are enough integers $m$ such that $f$ permutes $\mathbb Z/m\mathbb Z$, and on the other hand $V(p)/p$ with $V(p)<p$ is bounded
away from 1 (thus also shows a gap directly below 1), then the polynomial discriminator can
be easily described for all $n$ large enough. See Moree 
\cite{M} and Zieve \cite{Z} for 
details.

\section*{Acknowledgments}
Part of this paper was written during an one month
internship in the autumn of 2016 of B.F. at the Max Planck
Institute for Mathematics in Bonn. She thanks the people of this institution for their hospitality. She was partly supported by the government of Canada's International Development Research Centre (IDRC) within the framework of the AIMS Research for Africa Project.\hfil\break
\indent Work on the paper was continued
during a visit of F.L. to MPIM in the
first half of 2017.
\par The authors like to thank
Alexandru Ciolan for help with 
computer experiments and proofreading earlier versions.

\end{document}